\documentclass{amsart}
\usepackage{amssymb,latexsym,amsmath,amscd,graphicx,setspace,amsthm,verbatim,comment, stmaryrd,upgreek}
\usepackage{tikz,tkz-berge}
\usepackage[margin = 3 cm]{geometry} \input xy \xyoption{all}

\usetikzlibrary{shapes.geometric}

\usetikzlibrary{decorations.markings}

\tikzset{every loop/.style={min distance=10mm,looseness=10}}

\newcount\mycount

\theoremstyle{plain}
\newtheorem{theorem}{Theorem}[section]
\newtheorem{proposition}[theorem]{Proposition}

\newtheorem{corollary}[theorem]{Corollary}
\newtheorem{lemma}[theorem]{Lemma}

\makeatletter
\def\th@remark{%
  \thm@headfont{\bfseries}%
  \normalfont 
  \thm@preskip\topsep \divide\thm@preskip\tw@
  \thm@postskip\thm@preskip
}
\makeatother

\theoremstyle{remark} 
\newtheorem{remark}[theorem]{Remark}
      
\theoremstyle{definition}
\newtheorem{definition}[theorem]{Definition}

\def\Z{\mathbb{Z}}
\def\N{\mathbb{N}}

\begin{document} 
\title[On abelian  $\ell$-towers of multigraphs III]{On abelian  $\ell$-towers of multigraphs III}

\author{Kevin McGown, Daniel Valli\`{e}res}

\address{Mathematics and Statistics Department, California State University, Chico, CA 95929 USA} 
\email{kmcgown@csuchico.edu} 
\address{Mathematics and Statistics Department, California State University, Chico, CA 95929 USA}
\email{dvallieres@csuchico.edu}

\date{\today}

\begin{abstract}
Let $\ell$ be a rational prime.  Previously, abelian $\ell$-towers of multigraphs were introduced which are analogous to $\Z_{\ell}$-extensions of number fields.  It was shown that for towers of bouquets, the growth of the $\ell$-part of the number of spanning trees behaves in a predictable manner (analogous to a well-known theorem of Iwasawa for $\Z_{\ell}$-extensions of number fields).  In this paper, we extend this result to abelian $\ell$-towers over an arbitrary connected multigraph (not necessarily simple and not necessarily regular).
In order to carry this out, we employ integer-valued polynomials to construct power series with coefficients in $\Z_\ell$ arising from cyclotomic number fields, different than the power series appearing in the prequel.  This allows us to study the special value at $u=1$ of the Artin--Ihara $L$-function, when the base multigraph is not necessarily a bouquet.
\end{abstract} 
\maketitle 
\tableofcontents 
\tikzset{->-/.style={decoration={
  markings,
  mark=at position #1 with {\arrow{>}}},postaction={decorate}}}
\section{Introduction}

In \cite{Vallieres:2021}, abelian $\ell$-towers of multigraphs, where $\ell$ is a rational prime, were introduced which can be viewed as being analogous to $\mathbb{Z}_{\ell}$-extensions of number fields.  To every tuple in $\Z_\ell^t$ (with $t\in\N$ and not all entries divisible by $\ell$) one can associate an abelian $\ell$-tower of a bouquet with $t$ loops.  It was proved in \cite{McGown/Vallieres:2021} that the $\ell$-adic valuation of the number of spanning trees behaves similarly to the $\ell$-adic valuation of the class numbers in $\mathbb{Z}_{\ell}$-extensions of number fields.  More precisely, if $\kappa_n$ denotes the number of spanning trees of the multigraph at the $n$-th level, then there exist non-negative integers $\mu, \lambda, n_{0}$ and an integer $\nu$ such that
$${\rm ord}_{\ell}(\kappa_n)=\mu\ell^n+\lambda n + \nu,$$ 
for $n\geq n_0$.  Meanwhile, in \cite{Gonet:2021}, the same behavior was shown to be true with a different method in the situation where the base multigraph is an arbitrary simple graph, meaning that it does not contain loops and parallel edges.  In this paper, we extend this result to abelian $\ell$-towers over arbitrary base multigraphs (not necessarily simple and not necessarily regular) by adapting the strategy of \cite{McGown/Vallieres:2021} and \cite{Vallieres:2021} to the current more general situation. 

The paper is organized as follows.  In \S \ref{integer_val}, we remind the reader about integer-valued polynomials.  We then use integer-valued polynomials in \S \ref{power_ser} to construct some power series in $\mathbb{Z}_{\ell} \llbracket T\rrbracket$ arising from cyclotomic number fields which will be used to calculate the invariants $\mu$ and $\lambda$.  In \S \ref{voltage}, we remind the reader about voltage assignments which give a convenient way of constructing abelian $\ell$-towers of multigraphs.  One can view the derived multigraphs obtained from voltage assignments as generalizations of the Cayley-Serre multigraphs used in \cite{McGown/Vallieres:2021} and \cite{Vallieres:2021}.  In \S \ref{ihara}, we study the special value at $u=1$ of the Artin-Ihara $L$-functions arising in abelian $\ell$-towers of multigraphs.  In \S \ref{main_th}, we formulate our main result (see Theorem \ref{maintheorem}), and we present a few examples in \S \ref{examples}.  At last, in \S \ref{Gonet} we point out the connection between our work and the work of Gonet included in \cite{Gonet:2021}.


\section{Integer-valued polynomials} \label{integer_val}
Our main reference for this section is \cite{Cahen/Chabert:2016}.  Recall that the set of integer-valued polynomials is
$${\rm Int}(\mathbb{Z}) = \{P(X) \in \mathbb{Q}[X] \, | \, P(\mathbb{Z}) \subseteq \mathbb{Z}\}. $$
It is a unital commutative ring.  As usual, for $n \ge 2$, we let
$$\binom{X}{n} = \frac{X(X-1) \ldots (X - n +1)}{n!} \in \mathbb{Q}[X],$$
and we set $\binom{X}{0} = 1$ and $\binom{X}{1}=X$.  It is known that $\binom{X}{n} \in {\rm Int}(\mathbb{Z})$ and that every integer-valued polynomial $P(X)$ of degree $n$ can be uniquely written as
$$P(X) = \sum_{k=0}^{n}c_{k}\binom{X}{k},$$
for some $c_{k} \in \mathbb{Z}$ that can be written down explicitly.  We refer to \S $1$ of \cite{Cahen/Chabert:2016} for details.

\begin{proposition} \label{Lipschitz}
Let $\ell$ be a rational prime and let $P(X) \in {\rm Int}(\mathbb{Z})$.  Then the induced function $P:\mathbb{Z} \longrightarrow \mathbb{Z}$ is Lipschitz when $\mathbb{Z}$ is endowed with the $\ell$-adic topology.  More precisely, we have 
$$|P(m) - P(n)|_{\ell} \le C(\ell,P) \cdot |m-n|_{\ell}, $$
for all $m,n \in \mathbb{Z}$, where
$$C(\ell,P) =  \ell^{\lfloor \log_{\ell}({\rm deg}(P))\rfloor}.$$
\end{proposition}
\begin{proof}
See Proposition $9$ in \cite{Cahen/Chabert:2016}.
\end{proof}

\section{Some power series arising from cyclotomic number fields} \label{power_ser}
In \cite{McGown/Vallieres:2021}, for each $a \in \mathbb{Z}_{\ell}$, a power series $P_{a}(T) \in \mathbb{Z}_{\ell}\llbracket T\rrbracket$ obtained as an $\ell$-adic limit of some shifted Chebyshev polynomials was constructed in order to prove the main result.  (Theorem $4.1$ in \cite{McGown/Vallieres:2021}.)  Using Proposition \ref{Lipschitz}, the main idea of \S $2$ and \S $3$ of \cite{McGown/Vallieres:2021} can be adapted to construct other power series which we find more convenient to work with in the present paper.  (In fact, Proposition \ref{Lipschitz} could be used to prove Corollary $3.1$ of \cite{McGown/Vallieres:2021} instead of using Lemma $2.2$ and Proposition $2.3$ of \cite{McGown/Vallieres:2021}.)

Given $m \in \mathbb{N} = \{1,2,\ldots \}$, we let $\zeta_{m} = {\rm exp}(2 \pi i/m) \in \mathbb{C}$.  If $a \in \mathbb{Z}$, we set
$$ \pi_{m}(a) = 1 - \zeta_{m}^{a} \in \overline{\mathbb{Q}} \subseteq \mathbb{C},$$
and we write $\pi_{m}$ rather than $\pi_{m}(1)$.
Note that $\pi_{m}(a)$ is an element of the cyclotomic number field $\mathbb{Q}(\zeta_{m})$.  Let now $\ell$ be a rational prime, $i \in \mathbb{N}$, and consider the cyclotomic number field $\mathbb{Q}(\zeta_{\ell^{i}})$.  Above $\ell$, there is a unique prime ideal which we denote by $\mathcal{L}_{i}$.  Furthermore, it is known that $\mathcal{L}_{i}$ is a principal ideal and 
\begin{equation} \label{principal}
\mathcal{L}_{i} = (\pi_{\ell^{i}}).
\end{equation}
For $a \in \mathbb{N}$, consider 
$$Q_{a}(T) = 1 - (1-T)^{a} \in \mathbb{Z}[T],$$
and set $Q_{0}(T) = 0$.  Then, 
\begin{equation} \label{eq}
Q_{a}(\pi_{m}) = \pi_{m}(a), 
\end{equation}
for all $m \in \mathbb{N}$ and for all $a \in \mathbb{Z}_{\ge 0 }$.  Note that the polynomials $Q_{a}(T)$ have no constant coefficient and have degree $a$.  For $k,n \in \mathbb{N}$, we let
\begin{equation*}
e_{k}(n)=
\begin{cases}
(-1)^{k-1}\binom{n}{k},&\text{ if } k \le n;\\
0,&\text{ otherwise}.
\end{cases}
\end{equation*}
Using the binomial expansion of $(1-T)^{a}$, we have
$$Q_{a}(T) = e_{1}(a)T + e_{2}(a)T^{2} + \ldots + e_{a}(a) T^{a}. $$
We note in passing that
$$e_{1}(n) = n \text{ and } e_{n}(n) = (-1)^{n-1}. $$

We now consider the complete local ring $\mathbb{Z}_{\ell}\llbracket T\rrbracket$ and the function $g:\mathbb{N} \longrightarrow \mathbb{Z}_{\ell}\llbracket T \rrbracket$ defined via
$$n \mapsto g(n) = Q_{n}(T). $$

\begin{proposition}
The function $g$ is uniformly continuous when $\mathbb{N}$ is endowed with the $\ell$-adic topology.
\end{proposition}
\begin{proof}
Proposition \ref{Lipschitz} implies that if $P(X) \in {\rm Int}(\mathbb{Z})$, then
$${\rm ord}_{\ell}(P(m) - P(n)) \ge {\rm ord}_{\ell}(m-n) - \lfloor \log_{\ell}({\rm deg}(P))\rfloor, $$
for all $m,n \in \mathbb{N}$.  Noting that $\lfloor \log_{\ell}(k) \rfloor \le k$ for $k \ge 1$, we have in fact
$${\rm ord}_{\ell}(P(m) - P(n)) \ge {\rm ord}_{\ell}(m-n) - {\rm deg}(P).$$
Consider now
$$e_{k}(X) =  (-1)^{k-1}\binom{X}{k} \in {\rm Int}(\mathbb{Z}).$$
Note that ${\rm deg}(e_{k}(X)) = k$; thus, if $m,n \in \mathbb{N}$ are such that
$$ m \equiv n \pmod{\ell^{s}},$$
for some $s \in \mathbb{N}$, then 
$$e_{k}(m) \equiv e_{k}(n) \pmod{\ell^{s-k}}$$
for $k=1,\ldots, s-1$.  So given any $N \in \mathbb{N}$ and $m,n \in \mathbb{N}$ satisfying
$${\rm ord}_{\ell}(m - n) \ge N,$$
we have
$$Q_{m}(T) - Q_{n}(T) \in \mathfrak{m}^{N},$$
where $\mathfrak{m}$ is the maximal ideal of $\mathbb{Z}_{\ell}\llbracket T \rrbracket$, and this shows the claim.
\end{proof}
Just as in \S$3$ of \cite{McGown/Vallieres:2021}, since $\mathbb{N}$ is dense in $\mathbb{Z}_{\ell}$, the function $g$ can be uniquely extended to a continuous function 
$$g:\mathbb{Z}_{\ell} \longrightarrow \mathbb{Z}_{\ell} \llbracket T\rrbracket, $$
which we denote by the same symbol.  If $a \in \mathbb{Z}_{\ell}$, we let 
$$Q_{a}(T) = g(a) \in \mathbb{Z}_{\ell}\llbracket T \rrbracket. $$

We fix an embedding $\tau:\mathbb{Q}(\zeta_{\ell^{\infty}})\hookrightarrow \overline{\mathbb{Q}}_{\ell}$ and for $i=0,1,\ldots \,$, we let $\xi_{\ell^{i}} = \tau(\zeta_{\ell^{i}}) \in \overline{\mathbb{Q}}_{\ell}$.  Furthermore for $a \in \mathbb{Z}$, we let
$$\rho_{\ell^{i}}(a) = \tau(\pi_{\ell^{i}}(a)) = 1-\xi_{\ell^{i}}^{a} \in \overline{\mathbb{Q}}_{\ell}, $$
and we write $\rho_{\ell^{i}}$ rather than $\rho_{\ell^{i}}(1)$.  We denote the valuation on $\mathbb{C}_{\ell}$ by $v_{\ell}$, normalized so that $v_{\ell}(\ell)=1$, and the absolute value on $\mathbb{C}_{\ell}$ by $|\cdot|_{\ell}$, normalized so that $v_{\ell}(x) = - \log_{\ell} (|x|_{\ell})$ for all $x \in \mathbb{C}_{\ell}$.  The relationship between $v_{\ell}$ and the valuation ${\rm ord}_{\mathcal{L}_{i}}$ on $\mathbb{Q}(\zeta_{\ell^{i}})$ associated to the prime ideal $\mathcal{L}_{i}$ is
\begin{equation} \label{relationship_val}
v_{\ell}(\tau(x)) = \frac{1}{\varphi(\ell^{i})} {\rm ord}_{\mathcal{L}_{i}}(x), 
\end{equation}
for all $x \in \mathbb{Q}(\zeta_{\ell^{i}})$, where $\varphi$ is the Euler $\varphi$-function.  In particular, by (\ref{principal}), we have
\begin{equation} \label{val}
v_{\ell}(\rho_{\ell^{i}}) = \frac{1}{\varphi(\ell^{i})} \text{ and } |\rho_{\ell^{i}}|_{\ell} = \ell^{-1/\varphi(\ell^{i})}.
\end{equation}

The same argument as in \S $3$ of \cite{McGown/Vallieres:2021} shows that for each $i \in \mathbb{N}$, we have a continuous function $\mathbb{Z}_{\ell} \longrightarrow \mathbb{C}_{\ell}$ defined via
$$a \mapsto \rho_{\ell^{i}}(a) = 1 - \xi_{\ell^{i}}^{a}, $$
and we note that in fact $\rho_{\ell^{i}}(\mathbb{Z}_{\ell}) \subseteq \overline{\mathbb{Q}}_{\ell}$.  From now on, we let
$$D = \{x \in \mathbb{C}_{\ell} \, : \, |x|_{\ell} < 1\}.$$
\begin{lemma} \label{cont}
Let $x \in D$.  The function ${\rm ev}_{x}:\mathbb{Z}_{\ell}\llbracket T \rrbracket \longrightarrow \mathbb{C}_{\ell}$ given by
$$Q(T) \mapsto {\rm ev}_{x}(Q(T)) = Q(x) $$
is uniformly continuous.
\end{lemma}
\begin{proof}
The proof is the same as the one for Lemma $3.2$ in \cite{McGown/Vallieres:2021}.
\end{proof}
We now obtain the following result.
\begin{proposition} \label{hip}
Given $a \in \mathbb{Z}_{\ell}$ and any $i \in \mathbb{Z}_{\ge 0 }$, we have $Q_{a}(\rho_{\ell^{i}}) = \rho_{\ell^{i}}(a)$.
\end{proposition}
\begin{proof}
If $i=0$, the equality is clear.  If $i \ge 1$, then the functions $\mathbb{Z}_{\ell} \longrightarrow \mathbb{C}_{\ell}$ defined via
$$a \mapsto \rho_{\ell^{i}}(a) \text{ and } a \mapsto {\rm ev}_{\rho_{\ell^{i}}} \circ g (a) $$
are continuous by (\ref{val}) and Lemma \ref{cont}.  They both agree on $\mathbb{N}$ by (\ref{eq}).  The result follows as once since $\mathbb{N}$ is dense in $\mathbb{Z}_{\ell}$.  
\end{proof}

\begin{remark} \label{difptview}
We note in passing that we can think about the power series $Q_{a}(T)$ as follows.  Consider the unital commutative ring ${\rm Int}(\mathbb{Z})\llbracket T \rrbracket$.  Then, we have
$$Q_{X}(T) = \sum_{k=1}^{\infty} e_{k}(X)T^{k} \in {\rm Int}(\mathbb{Z})\llbracket T \rrbracket. $$
By Proposition \ref{Lipschitz}, we have that ${\rm Int}(\mathbb{Z})$ is a subring of the ring $\mathcal{C}(\mathbb{Z}_{\ell},\mathbb{Z}_{\ell})$ of $\mathbb{Z}_{\ell}$-valued continuous functions on $\mathbb{Z}_{\ell}$ for all rational primes $\ell$.  Thus, we also have
$$Q_{X}(T) \in \mathcal{C}(\mathbb{Z}_{\ell},\mathbb{Z}_{\ell}) \llbracket T \rrbracket. $$
The power series $Q_{a}(T)$ for $a \in \mathbb{Z}_{\ell}$ are obtained by setting $X=a$, that is
$$Q_{a}(T)  = \sum_{k=1}^{\infty} e_{k}(a)T^{k} \in \mathbb{Z}_{\ell} \llbracket T\rrbracket. $$
On the other hand, the power series $P_{a}(T) \in \mathbb{Z}_{\ell}\llbracket T\rrbracket$ arising in \cite{McGown/Vallieres:2021} are related to
$$P_{X}(T) = \sum_{k=1}^{\infty}d_{k}(X)T^{k} \in {\rm Int}_{even}(\mathbb{Z})\llbracket T \rrbracket,$$
where ${\rm Int}_{even}(\mathbb{Z})$ is the subring of ${\rm Int}(\mathbb{Z})$ consisting of even integer-valued polynomials, and 
$$d_{k}(X) = (-1)^{k-1}\binom{X+k-1}{2k-1}\frac{X}{k}. $$
Note that the $d_{k}(X)$ (along with $\binom{X}{0}$) form a $\mathbb{Z}$-basis for ${\rm Int}_{even}(\mathbb{Z})$.  (See problem $89$ of Chapter $2$, Part Eight in \cite{Polya/Szego:1998}.)
\end{remark}
\begin{remark}
For the following discussion, let $\ell$ be an odd rational prime and recall the identity
$$1 - X^{\ell} = \prod_{k=0}^{\ell -1}(1 - \xi_{\ell}^{k}X). $$
Thus, if $a \in \mathbb{Z}_{\ell}^{\times}$ and $n \in \mathbb{N}$, then
$$1 - (\xi_{\ell^{n}}^{a})^{\ell} =   \prod_{k=0}^{\ell -1}(1 - \xi_{\ell}^{k}\xi_{\ell^{n}}^{a}).$$
Let now $F_{n} = \mathbb{Q}_{\ell}(\xi_{\ell^{n}})$ and consider the norm map $N_{n}:F_{n} \longrightarrow F_{n-1}$ for $n \ge 2$.  We have
\begin{equation} \label{norm}
N_{n}(\rho_{\ell^{n}}(a)) = \prod_{k=0}^{\ell -1}(1 - \xi_{\ell}^{k}\xi_{\ell^{n}}^{a}) = \rho_{\ell^{n-1}}(a).  
\end{equation}
Letting
$$u_{n}(a) = \frac{\rho_{\ell^{n}}(a)}{ \rho_{\ell^{n}}}, $$
we have $u_{n}(a) \in E(F_{n})$, where $E(F_{n})$ is the group of units of the ring of integers of $F_{n}$.  By  (\ref{norm}), we get
$$N_{n}(u_{n}(a)) = u_{n-1}(a). $$
Therefore, the power series
$$R_{a}(T) = \frac{Q_{a}(T)}{T} \in \mathbb{Z}_{\ell} \llbracket T \rrbracket^{\times} $$
satisfies
$$R_{a}(\rho_{\ell^{n}}) = u_{n}(a), $$
for all $n \ge 1$.  It follows that the power series $R_{a}(T)$, when $a \in \mathbb{Z}_{\ell}^{\times}$, are particular examples of the ones studied in \cite{Coleman:1979}.

\end{remark}

\section{Voltage assignments and their derived multigraphs} \label{voltage}
We recall from \S $2.2$ of \cite{Vallieres:2021} how we think about multigraphs, but our notation will be slightly different.  Our main reference is \cite{Sunada:2013}.  A directed multigraph $X$ consists of a set of vertices $V_{X}$ and a set of directed edges $\vec{E}_{X}$ with a function
$${\rm inc}:\vec{E}_{X} \longrightarrow V_{X} \times V_{X}$$
called the incidence map.  For $e \in \vec{E}_{X}$, we write
$${\rm inc}(e) = (o(e),t(e)),$$
so that we obtain two functions $o,t:\vec{E}_{X} \longrightarrow V_{X}$ called the origin and the terminus maps.  An undirected multigraph $X$, or simply a multigraph, consists of a directed multigraph as above with another function
$${\rm inv}:\vec{E}_{X} \longrightarrow \vec{E}_{X}$$
called the inversion map which satisfies
\begin{enumerate}
\item ${\rm inv}^{2} = {\rm id}_{\vec{E}_{X}}$,
\item ${\rm inv}(e) \neq e$ for all $e \in \vec{E}_{X}$,
\item ${\rm inc}({\rm inv}(e)) = \iota({\rm inc}(e))$ for all $e \in \vec{E}_{X}$,
\end{enumerate}
where $\iota:V_{X} \times V_{X} \longrightarrow V_{X} \times V_{X}$ is defined via $\iota(v,w) = (w,v)$.  If $v \in V_{X}$, we let
$$E_{v} = \{e \in \vec{E}_{X} \, | \, o(e) = v \}, $$
and the valency (or degree) of a vertex is defined to be 
$${\rm val}_{X}(v) = |E_{v}|. $$
The quotient $E_{X} = \vec{E}_{X}/\langle {\rm inv} \rangle$ is the set of undirected edges.  A multigraph is called finite if both $V_{X}$ and $E_{X}$ are finite sets.  \emph{All multigraphs arising in this paper will be finite and without any vertex of valency one, unless otherwise stated.}

Note that given a directed multigraph, one can always obtain an (undirected) multigraph by ``forgetting'' the directions.  Formally, one adds for each directed edge another directed edge in the opposite direction and the map ${\rm inv}$ sends a directed edge to this new directed edge with opposite direction.  The details are left to the reader.

We now explain the notion of voltage assignment and their derived multigraphs in a way that is convenient for us.  Our main reference for voltage assignments is Chapter $2$ of \cite{Gross/Tucker:2001}.  Let $X$ be a connected multigraph and pick a section $\omega:E_{X} \longrightarrow \vec{E}_{X}$ of the natural map $\vec{E}_{X} \longrightarrow E_{X}$.  In other words, we fix a direction for each of the undirected edges of $X$.  Let 
$$S = \omega (E_{X}) \subseteq \vec{E}_{X}.$$  
If $G$ is an additive finite group, a voltage assignment with values in $G$ is a function
$$\alpha:S \longrightarrow G. $$
To each such voltage assignment, one can associate a derived multigraph which we denote by $X(G,S,\alpha)$.  The multigraph $X(G,S,\alpha)$ is obtained as follows.  The vertices of $X(G,S,\alpha)$ consist of $V_{X} \times G$ and the edges consist of $S \times G$.  An edge $(s,\sigma) \in S \times G$ connects the vertex $(o(s),\sigma)$ to the vertex $(t(s),\sigma + \alpha(s))$.  In this way, we get a directed multigraph, but we forget about the directions and think of $X(G,S,\alpha)$ as a multigraph.  Note that the choice of directions for the loops of the base multigraph does not affect the resulting multigraph $X(G,S,\alpha)$.  If $X$ is a bouquet, then the multigraphs $X(G,S,\alpha)$ are the Cayley-Serre multigraphs used in \cite{McGown/Vallieres:2021}  and \cite{Vallieres:2021}.

For instance, if we let $X$ be the dumbbell multigraph
\begin{center}
\begin{tikzpicture}
\node [circle, draw, fill=white] (b1) at (1,1) {};
\node [circle, draw, fill=white] (b2) at (2,1) {};

\node (c) at (1.16,1) {};
\fill (c) circle[radius=0.7pt];

\node (d) at (1.84,1) {};
\fill (d) circle[radius=0.7pt];

\draw (b1) -- (b2);
\end{tikzpicture},
\end{center}
the section be
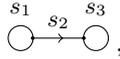
\begin{center}
\begin{tikzpicture}
\node [label = {\small $s_{1}$},circle, draw, fill=white] (b1) at (1,1) {};
\node [label = {\small $s_{3}$}, circle, draw, fill=white] (b2) at (2,1) {};

\node (c) at (1.16,1) {};
\fill (c) circle[radius=0.7pt];

\node (d) at (1.84,1) {};
\fill (d) circle[radius=0.7pt];

\draw[->-=.6] (b1) -- (b2) node [midway,above] {\small $s_{2}$};
\end{tikzpicture},
\end{center}
and the function $\alpha:S \longrightarrow \mathbb{Z}/5\mathbb{Z}$ be defined via $\alpha(s_{1}) = \bar{1}$, $\alpha(s_{2}) = \bar{0}$ and $\alpha(s_{3}) = \bar{2}$, then the graph $X(\mathbb{Z}/5\mathbb{Z},S,\alpha)$ is the Petersen graph
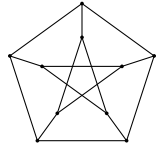
\begin{center}
\begin{tikzpicture}[baseline={([yshift=-0.6ex] current bounding box.center)}]
\node[draw=none,minimum size=2cm,regular polygon,regular polygon sides=5] (a) {};
\node[draw=none, minimum size=1.1cm,regular polygon,regular polygon sides=5] (b) {};

\foreach \x in {1,2,...,5}
  \fill (a.corner \x) circle[radius=0.7pt];
  
\foreach \y in {1,2,...,5}
  \fill (b.corner \y) circle[radius=0.7pt];
  
\path (a.corner 1) edge (a.corner 2);
\path (a.corner 2) edge (a.corner 3);
\path (a.corner 3) edge (a.corner 4);
\path (a.corner 4) edge (a.corner 5);
\path (a.corner 5) edge (a.corner 1);

\path (a.corner 1) edge (b.corner 1);
\path (a.corner 2) edge (b.corner 2);
\path (a.corner 3) edge (b.corner 3);
\path (a.corner 4) edge (b.corner 4);
\path (a.corner 5) edge (b.corner 5);

\path (b.corner 1) edge (b.corner 3);
\path (b.corner 1) edge (b.corner 4);
\path (b.corner 2) edge (b.corner 4);
\path (b.corner 2) edge (b.corner 5);
\path (b.corner 3) edge (b.corner 5);

\end{tikzpicture}
\end{center}

From now on, if $X$ is a connected multigraph and we say ``Let $\alpha:S \longrightarrow G$ be a function\ldots'', it will be understood that a section $\omega$ has been chosen and that $S = \omega(E_{X})$.  Now, if we start with a connected multigraph $X$ and a function $\alpha:S \longrightarrow \mathbb{Z}_{\ell}$, then for each $n \in \mathbb{N}$, we consider the voltage assignment
$$\alpha_{n}:S \longrightarrow \mathbb{Z}/\ell^{n}\mathbb{Z}$$
obtained from the composition
$$S \stackrel{\alpha}{\longrightarrow} \mathbb{Z}_{\ell} \longrightarrow \mathbb{Z}_{\ell}/\ell^{n} \mathbb{Z}_{\ell} \stackrel{\simeq}{\longrightarrow} \mathbb{Z}/\ell^{n}\mathbb{Z}.$$
If we assume that all the multigraphs $X(\mathbb{Z}/\ell^{n}\mathbb{Z},S,\alpha_{n})$ are connected, then we get an abelian $\ell$-tower of multigraphs
$$X  \longleftarrow X(\mathbb{Z}/\ell \mathbb{Z},S,\alpha_{1}) \longleftarrow X(\mathbb{Z}/\ell^{2}\mathbb{Z},S,\alpha_{2}) \longleftarrow \ldots \longleftarrow X(\mathbb{Z}/\ell^{n}\mathbb{Z},S,\alpha_{n}) \longleftarrow \ldots$$
in the sense of Definition $4.1$ of \cite{Vallieres:2021}.

For example, if we take $\ell = 5$, $X$ and $S$ as above, and the function $\alpha:S \longrightarrow \mathbb{Z} \subseteq \mathbb{Z}_{5}$ defined by $\alpha(s_{1}) = 1$, $\alpha(s_{2}) = 0$, and $\alpha(s_{3}) = 2$, then we have
$$X(\mathbb{Z}/5^{n}\mathbb{Z},S,\alpha_{n}) = G(5^{n},2),$$
where $G(5^{n},2)$ is the generalized Petersen graph introduced in \cite{Coxeter:1950} and \cite{Watkins:1969}.  We then get an abelian $5$-tower 
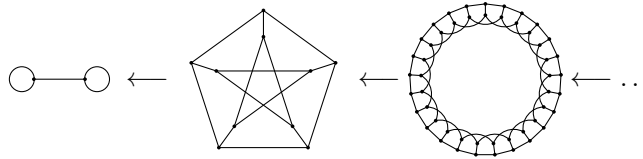
\begin{equation*}
\begin{tikzpicture}[baseline={([yshift=-0.6ex] current bounding box.center)}]
\node [circle, draw, fill=white] (b1) at (1,1) {};
\node [circle, draw, fill=white] (b2) at (2,1) {};

\node (c) at (1.16,1) {};
\fill (c) circle[radius=0.7pt];

\node (d) at (1.84,1) {};
\fill (d) circle[radius=0.7pt];

\draw (b1) -- (b2);
\end{tikzpicture}
\, \, \longleftarrow \, \, \,
\begin{tikzpicture}[baseline={([yshift=-0.6ex] current bounding box.center)}]
\node[draw=none,minimum size=2cm,regular polygon,regular polygon sides=5] (a) {};
\node[draw=none, minimum size=1.3cm,regular polygon,regular polygon sides=5] (b) {};

\foreach \x in {1,2,...,5}
  \fill (a.corner \x) circle[radius=0.7pt];
  
\foreach \y in {1,2,...,5}
  \fill (b.corner \y) circle[radius=0.7pt];
  
\path (a.corner 1) edge (a.corner 2);
\path (a.corner 2) edge (a.corner 3);
\path (a.corner 3) edge (a.corner 4);
\path (a.corner 4) edge (a.corner 5);
\path (a.corner 5) edge (a.corner 1);

\path (a.corner 1) edge (b.corner 1);
\path (a.corner 2) edge (b.corner 2);
\path (a.corner 3) edge (b.corner 3);
\path (a.corner 4) edge (b.corner 4);
\path (a.corner 5) edge (b.corner 5);

\path (b.corner 1) edge (b.corner 3);
\path (b.corner 1) edge (b.corner 4);
\path (b.corner 2) edge (b.corner 4);
\path (b.corner 2) edge (b.corner 5);
\path (b.corner 3) edge (b.corner 5);

\end{tikzpicture}
\, \, \, \longleftarrow 
\begin{tikzpicture}[baseline={([yshift=-0.6ex] current bounding box.center)}]
\node[draw=none,minimum size=2cm,regular polygon,regular polygon sides=25] (a) {};
\node[draw=none, minimum size=1.7cm,regular polygon,regular polygon sides=25] (b) {};

\foreach \x in {1,2,...,25}
  \fill (a.corner \x) circle[radius=0.7pt];
  
\foreach \y in {1,2,...,25}
  \fill (b.corner \y) circle[radius=0.7pt];
  
\foreach \x\z in {1/2,2/3,3/4,4/5,5/6,6/7,7/8,8/9,9/10,10/11,11/12,12/13,13/14,14/15,15/16,16/17,17/18,18/19,19/20,20/21,21/22,22/23,23/24,24/25,25/1}
  \path (a.corner \x) edge (a.corner \z);  
 
\foreach \x\z in {1/1,2/2,3/3,4/4,5/5,6/6,7/7,8/8,9/9,10/10,11/11,12/12,13/13,14/14,15/15,16/16,17/17,18/18,19/19,20/20,21/21,22/22,23/23,24/24,25/25} 
  \path (a.corner \x) edge (b.corner \z);
  
\foreach \x\z in {1/3,2/4,3/5,4/6,5/7,6/8,7/9,8/10,9/11,10/12,11/13,12/14,13/15,14/16,15/17,16/18,17/19,18/20,19/21,20/22,21/23,22/24,23/25,24/1,25/2} 
  \path (b.corner \x) edge [bend left=50] (b.corner \z);
  



\end{tikzpicture}
\longleftarrow \ldots
\end{equation*}
Note that since $X$ is $3$-regular, then so are $G(5^{n},2)$ for all $n \ge 1$.

\section{The special value at $u=1$ of Artin-Ihara $L$-functions} \label{ihara}
We now recall what we need to know about Artin-Ihara $L$-functions for our present purposes.  Our main reference for this section is \cite{Terras:2011}, but our notation will be the same as in \S $2$ of \cite{Vallieres:2021}.  Let $Y/X$ be an abelian cover of connected multigraphs and let $G = {\rm Gal}(Y/X)$.  If $\psi \in \widehat{G} = {\rm Hom}_{\mathbb{Z}}(G,\mathbb{C}^{\times})$, then we denote the corresponding Artin-Ihara $L$-function by $L_{X}(u,\psi)$.  Let us introduce a labeling of the vertices $V_{X} = \{v_{1},\ldots,v_{g} \}$ and for each $i=1,\ldots,g$, let us fix a vertex $w_{i} \in V_{Y}$ above $v_{i}$.  Recall the following definition (Definition $18.13$ in \cite{Terras:2011}).
\begin{definition} \label{artinian}
With the same notation as above,
\begin{enumerate}
\item For $\sigma \in G$, we let $A(\sigma)=(a_{ij}(\sigma))$ be the $g \times g$ matrix defined via
\begin{equation*}
a_{ij}(\sigma) = 
\begin{cases}
\text{Twice the number of undirected loops at the vertex }w_{i}, &\text{ if } i=j \text{ and } \sigma = 1;\\
\text{The number of undirected edges connecting $w_{i}$ to $w_{j}^{\sigma}$}, &\text{ otherwise}.
\end{cases}
\end{equation*}
\item If $\psi \in \widehat{G}$, then we let
$$A_{\psi} = \sum_{\sigma \in G} \psi(\sigma) \cdot A(\sigma). $$
\end{enumerate}
\end{definition}
If $D = (d_{ij})$ denotes the valency matrix attached to $X$, that is $D$ is diagonal and $d_{ii} = {\rm val}_{X}(v_{i})$, then the three-term determinant formula for the Artin-Ihara $L$-function (Theorem $18.15$ in \cite{Terras:2011}) gives
$$\frac{1}{L_{X}(u,\psi)} = (1-u^{2})^{-\chi(X)} \cdot {\rm det}(I - A_{\psi}u + (D-I)u^{2}),$$
where $\chi(X)$ is the Euler characteristic of $X$.  As in \cite{Vallieres:2021}, we let
\begin{equation*} 
h_{X}(u,\psi) = {\rm det}(I - A_{\psi}u + (D-I)u^{2}) \in \mathbb{Z}[\psi][u], 
\end{equation*}
where $\mathbb{Z}[\psi]$ is the ring of integers in the cyclotomic number field $\mathbb{Q}(\psi)$.  Evaluating at $u=1$ gives
\begin{equation}\label{mot}
h_{X}(1,\psi) = {\rm det}(D - A_{\psi}) \in \mathbb{Z}[\psi].
\end{equation}  
We let the absolute Galois group $G_{\mathbb{Q}}$ of $\mathbb{Q}$ acts on $\widehat{G}$ as usual.  Assuming $\chi(X) \neq 0$, equation ($7$) in \cite{Vallieres:2021} shows that
$$|G| \cdot \kappa_{Y} = \kappa_{X} \prod_{\Psi \neq \Psi_{0}} h_{X}(1,\Psi), $$
where $\kappa$ denotes the number of spanning trees of a multigraph, the product is over all non-trivial orbits $\Psi \in G_{\mathbb{Q}} \backslash \widehat{G}$, and 
$$h_{X}(1,\Psi) = \prod_{\psi \in \Psi} h_{X}(1,\psi) \in \mathbb{Z}. $$

\emph{From now on, we assume that $\chi(X) \neq 0$.}  (The case where $\chi(X) = 0$ was treated separately.  See the discussion after Definition $4.1$ of \cite{Vallieres:2021}.)
If 
$$X = X_{0} \longleftarrow X_{1} \longleftarrow X_{2} \longleftarrow \ldots \longleftarrow X_{n} \longleftarrow \ldots$$
is an abelian $\ell$-tower consisting of connected multigraphs above $X$, then for each $i=1,2,\ldots$, we let $G_{i} = {\rm Gal}(X_{i}/X) \simeq \mathbb{Z}/\ell^{i}\mathbb{Z}$.  Furthermore, we let $\psi_{i} \in \widehat{G}_{i}$ be the faithful character satisfying $\psi_{i}(\bar{1}) = \zeta_{\ell^{i}}$.  Equation ($12$) of \cite{Vallieres:2021} adapted to the case where $\kappa_{X} \neq 1$ gives
\begin{equation} \label{key_formula}
{\rm ord}_{\ell}(\kappa_{n}) = {\rm ord}_{\ell}(\kappa_{X})-n + \sum_{i=1}^{n}{\rm ord}_{\mathcal{L}_{i}}(h_{X}(1,\psi_{i})),
\end{equation}
where $\kappa_{n}$ is the number of spanning trees of $X_{n}$.

\begin{lemma}\label{key_prop}
Let $X$ be a connected multigraph and $G$ an additive finite group.  Let $\alpha:S \longrightarrow G$ be a voltage assignment for which $X(G,S,\alpha)$ is connected.  Choose a labeling of the vertices of $X$, say $V_{X} = \{ v_{1},\ldots,v_{g}\}$ and let $w_{i} = (v_{i},0_{G}) \in V_{X(G,S,\alpha)}$, where $0_{G}$ is the neutral element of $G$.  For each $i,j \in \{1,\ldots,g \}$ and $\sigma \in G$, let
$$b_{ij}(\sigma) = |\{s \in S \, | \, {\rm inc}(s) = (v_{i},v_{j}) \text{ and } \alpha(s) = \sigma \}|$$
and
$$c_{ij}(\sigma) = |\{s \in S \, | \, {\rm inc}(s) = (v_{j},v_{i}) \text{ and } \alpha(s) = - \sigma \}|. $$
Then, the entries $a_{ij}(\sigma)$ of the matrix $A(\sigma)$ of Definition \ref{artinian} satisfy
$$a_{ij}(\sigma) = b_{ij}(\sigma) + c_{ij}(\sigma). $$
\end{lemma}
\begin{proof}
This follows from the definition of the derived multigraphs $X(G,S,\alpha)$ given in \S \ref{voltage}.
\end{proof}
\begin{corollary} \label{co}
With the same notation as in Lemma \ref{key_prop}, the matrix $A_{\psi}$ of Definition \ref{artinian} satisfies
$$A_{\psi} = \left(\sum_{\substack{s \in S \\ {\rm inc}(s) = (v_{i},v_{j})}}\psi(\alpha(s)) + \sum_{\substack{s \in S \\ {\rm inc}(s) = (v_{j},v_{i})}}\psi(-\alpha(s)) \right). $$
\end{corollary}
\begin{proof}
The $ij$-th entry of $A_{\psi}$ is given by
\begin{equation*}
\begin{aligned}
\sum_{\sigma \in G} \psi(\sigma) a_{ij}(\sigma) &= \sum_{\sigma \in G}\psi(\sigma)b_{ij}(\sigma) + \sum_{\sigma \in G}\psi(\sigma)c_{ij}(\sigma) \\
&=\sum_{\substack{s \in S \\ {\rm inc}(s) = (v_{i},v_{j})}}\psi(\alpha(s)) + \sum_{\substack{s \in S \\ {\rm inc}(s) = (v_{j},v_{i})}}\psi(-\alpha(s)),
\end{aligned}
\end{equation*}
by Lemma \ref{key_prop}, and this is what we wanted to show.
\end{proof}
\begin{remark}
If $X$ is a bouquet, then the last corollary specializes to Theorem $5.2$ of \cite{Vallieres:2021}.
\end{remark}
We will now apply Lemma \ref{key_prop} and Corollary \ref{co} to the successive layers in an abelian $\ell$-tower of multigraphs.
\begin{proposition} \label{useful}
Let $X$ be a connected multigraph, $\ell$ a rational prime, and let $\alpha:S \longrightarrow \mathbb{Z}_{\ell}$ be a function for which all of the multigraphs $X(\mathbb{Z}/\ell^{n}\mathbb{Z},S,\alpha_{n})$ are connected.  Label the vertices $V_{X} = \{ v_{1},\ldots,v_{g}\}$.  For each $n \ge 1$, we let
$$w_{i,n} = (v_{i},\bar{0}) \in V_{X(\mathbb{Z}/\ell^{n}\mathbb{Z},S,\alpha_{n})}. $$  
For $i,j \in \{ 1,\ldots,g\}$, let us define the integers
$$B_{ij} = |\{s \in S \, | \, {\rm inc}(s) = (v_{i},v_{j}) \} |$$
and
$$C_{ij} =|\{ s \in S \, | \, {\rm inc}(s) = (v_{j},v_{i})\}|. $$
Then, for $k=0,1,2,\ldots$, we have
$$\tau(D - A_{\psi_{k}}) = \left( d_{ij} - B_{ij} - C_{ij} + \sum_{\substack{s \in S \\ {\rm inc}(s) = (v_{i},v_{j})}} \rho_{\ell^{k}}(\alpha(s)) +  \sum_{\substack{s \in S \\ {\rm inc}(s) = (v_{j},v_{i})}} \rho_{\ell^{k}}(-\alpha(s))\right). $$
\end{proposition}
\begin{proof}
For $k \in \mathbb{N}$, let $p_{k}:\mathbb{Z}_{\ell} \longrightarrow \mathbb{Z}$ be the function defined by
$$p_{k}\left(\sum_{i=0}^{\infty}a_{i} \ell^{i} \right) = \sum_{i=0}^{k-1} a_{i} \ell^{i}, \, \, (a_{i} \in \{0,\ldots, \ell - 1\}).$$
By Corollary \ref{co}, the $ij$-th entry of the matrix $D - A_{\psi_{k}}$ is 
\begin{equation*}
d_{ij} - \left( \sum_{\substack{s \in S \\ {\rm inc}(s) = (v_{i},v_{j})}}\psi_{k}(\alpha_{k}(s)) + \sum_{\substack{s \in S \\ {\rm inc}(s) = (v_{j},v_{i})}}\psi_{k}(-\alpha_{k}(s))\right)
\end{equation*}
which is equal to
$$d_{ij} - B_{ij} - C_{ij} + \sum_{\substack{s \in S \\ {\rm inc}(s) = (v_{i},v_{j})}}\pi_{\ell^{k}}(p_{k}(\alpha(s)) + \sum_{\substack{s \in S \\ {\rm inc}(s) = (v_{j},v_{i})}}\pi_{\ell^{k}}(-p_{k}(\alpha(s)).$$
Applying $\tau$ to this last equation gives the desired result.  If $k=0$, we get the Laplacian matrix of $X$ on both sides, so we still have an equality.
\end{proof}

\begin{corollary} \label{poly}
With the same notation as in Proposition \ref{useful}, write 
$$\alpha(S) = \{b_{1},\ldots,b_{t}\}$$ 
for some distinct $b_{y} \in \mathbb{Z}_{\ell}$.  For $i,j \in \{1,\ldots,g \}$ and $y \in \{1,\ldots,t\}$, define the integers
$$\gamma_{ij}^{(y)} = |\{s \in S \, | \, {\rm inc}(s) = (v_{i},v_{j}) \text{ and } \alpha(s) = b_{y} \}|, $$
$$\delta_{ij}^{(y)} = |\{s \in S \, | \, {\rm inc}(s) = (v_{j},v_{i}) \text{ and } \alpha(s) = b_{y} \}|,$$
and the homogenous linear polynomials
$$P_{ij}(X_{1},\ldots,X_{t},Y_{1},\ldots,Y_{t}) = \gamma_{ij}^{(1)}X_{1} + \ldots + \gamma_{ij}^{(t)}X_{t} + \delta_{ij}^{(1)}Y_{1} + \ldots + \delta_{ij}^{(t)}Y_{t}.$$
Furthermore, define the $g \times g$ matrix
$$M = (d_{ij} - B_{ij} - C_{ij} + P_{ij}(X_{1},\ldots,X_{t},Y_{1},\ldots,Y_{t})),$$
and the polynomial with integer coefficients
$$ P(X_{1},\ldots,X_{t},Y_{1},\ldots,Y_{t}) = {\rm det}(M).$$
Then
\begin{enumerate}
\item The polynomial $P$ has no constant term,
\item $\tau(h_{X}(1,\psi_{k})) = P(\rho_{\ell^{k}}(b_{1}),\ldots, \rho_{\ell^{k}}(b_{t}),\rho_{\ell^{k}}(-b_{1}),\ldots,\rho_{\ell^{k}}(-b_{t}))$ for all $k \ge 1$.
\end{enumerate}
\end{corollary}
\begin{proof}
Indeed, by Proposition \ref{useful}, we have
$${\rm det}(\tau(D - A_{\psi_{k}})) = P(\rho_{\ell^{k}}(b_{1}),\ldots, \rho_{\ell^{k}}(b_{t}),\rho_{\ell^{k}}(-b_{1}),\ldots,\rho_{\ell^{k}}(-b_{t})),$$
and thus the second claim follows from (\ref{mot}).

For the first claim, note that
$$P(0,\ldots,0) = {\rm det}(d_{ij} - B_{ij} - C_{ij}), $$
but the matrix $(d_{ij} - B_{ij} - C_{ij})$ is the Laplacian matrix of $X$ which is singular.
\end{proof}
\begin{remark}
If $X$ is a bouquet, then the polynomial $P$ of Corollary \ref{poly} has the simple form
$$P(X_{1},\ldots,X_{t},Y_{1},\ldots,Y_{t}) =  C_{1}(X_{1}+Y_{1}) + \ldots + C_{t}(X_{t} + Y_{t}),$$
for some integers $C_{i}$ ($i=1,\ldots,t$).
\end{remark}

\section{Abelian $\ell$-towers of multigraphs} \label{main_th}

We can now state our main result which can be viewed as a graph theoretical analogue of a theorem of Iwasawa.  (See Theorem $11$ in \cite{Iwasawa:1959} and \S $4.2$ of \cite{Iwasawa:1973}.)
\begin{theorem}\label{maintheorem}
Let $X$ be a connected multigraph satisfying $\chi(X) \neq 0$ and let $\alpha:S \longrightarrow \mathbb{Z}_{\ell}$ be a function for which all multigraphs $X(\mathbb{Z}/\ell^{n}\mathbb{Z},S,\alpha_{n})$ are connected.  Write
$$\alpha(S) = \{b_{1},\ldots,b_{t} \},$$
where the $b_{y}$ are distinct $\ell$-adic integers.  Consider the abelian $\ell$-tower
$$X  \longleftarrow X(\mathbb{Z}/\ell \mathbb{Z},S,\alpha_{1}) \longleftarrow X(\mathbb{Z}/\ell^{2}\mathbb{Z},S,\alpha_{2}) \longleftarrow \ldots \longleftarrow X(\mathbb{Z}/\ell^{n}\mathbb{Z},S,\alpha_{n}) \longleftarrow \ldots$$
and define the $\ell$-adic integers $c_{j}$ via 
\begin{equation*}
\begin{aligned}
Q(T) &= P(Q_{b_{1}}(T),\ldots,Q_{b_{t}}(T),Q_{-b_{1}}(T),\ldots,Q_{-b_{t}}(T)) \\
&=c_{1}T + c_{2}T^{2} + \ldots \in \mathbb{Z}_{\ell}\llbracket T \rrbracket,
\end{aligned}
\end{equation*}
where $P(X_{1},\ldots,X_{t},Y_{1},\ldots,Y_{t}) \in \mathbb{Z}[X_{1},\ldots,X_{t},Y_{1},\ldots,Y_{t}]$ is the polynomial from Corollary \ref{poly}.
Let
$$\mu = {\rm min}\{v_{\ell}(c_{j})\, | \, j =1,2,\ldots\}, $$
and 
$$\lambda = {\rm min}\{j \, | \, j \in \mathbb{N} \text{ and } v_{\ell}(c_{j}) = \mu\}  -1. $$
If $\kappa_{n}$ denotes the number of spanning trees of $X(\mathbb{Z}/\ell^{n}\mathbb{Z},S,\alpha_{n})$, then there exist a nonnegative integer $n_{0}$ and a constant $\nu \in \mathbb{Z}$ (depending also on the $b_{y}$) such that
$${\rm ord}_{\ell}(\kappa_{n}) = \mu \ell^{n} + \lambda n + \nu,$$
when $n \ge n_{0}$.  
\end{theorem}
\begin{proof}
Our starting point is the formula (\ref{key_formula}).  Combining with (\ref{relationship_val}) and Corollary \ref{poly} give
\begin{equation*}
\begin{aligned}
{\rm ord}_{\ell}(\kappa_{n}) &= {\rm ord}_{\ell}(\kappa_{X}) -n + \sum_{i=1}^{n}\varphi(\ell^{i})v_{\ell}\left(\tau(h_{X}(1,\psi_{i})) \right)  \\
&= {\rm ord}_{\ell}(\kappa_{X}) -n + \sum_{i=1}^{n}\varphi(\ell^{i})v_{\ell}\left(P(\rho_{\ell^{i}}(b_{1}),\ldots, \rho_{\ell^{i}}(b_{t}),\rho_{\ell^{i}}(-b_{1}),\ldots,\rho_{\ell^{i}}(-b_{t})) \right).
\end{aligned}
\end{equation*}
By Proposition \ref{hip}, we have
\begin{equation*}
\begin{aligned}
P(\rho_{\ell^{i}}(b_{1}),\ldots, \rho_{\ell^{i}}(b_{t}),\rho_{\ell^{i}}(-b_{1}),\ldots,\rho_{\ell^{i}}(-b_{t})) &= P(Q_{b_{1}}(\rho_{\ell^{i}}),\ldots,Q_{b_{t}}(\rho_{\ell^{i}}),Q_{-b_{1}}(\rho_{\ell^{i}}),\ldots, Q_{-b_{t}}(\rho_{\ell^{i}})) \\
&= Q(\rho_{\ell^{i}}).
\end{aligned}
\end{equation*}
The exact same argument as in the proof of Theorem $4.1$ of \cite{McGown/Vallieres:2021} shows that for $i$ large, we have
$$v_{\ell}(Q(\rho_{\ell^{i}})) = \mu + \frac{\lambda + 1}{\varphi(\ell^{i})}. $$
Therefore, there exists $n_{0} \ge 0$ and an integer $C$ such that if $n \ge n_{0}$, then
\begin{equation*}
\begin{aligned}
{\rm ord}_{\ell}(\kappa_{n}) &= {\rm ord}_{\ell}(\kappa_{X}) -n + C +  \sum_{i = n_{0}}^{n}\varphi(\ell^{i})v_{\ell}(Q(\rho_{\ell^{i}})) \\
&=  {\rm ord}_{\ell}(\kappa_{X})-n + C + \sum_{i = n_{0}}^{n}(\mu \cdot \varphi(\ell^{i}) + (\lambda+1))\\
&=  {\rm ord}_{\ell}(\kappa_{X})-n + C + (\lambda+1)(n-(n_{0}-1)) + \mu(\ell^{n} - \ell^{n_{0} - 1}),
\end{aligned}
\end{equation*}
and this ends the proof.
\end{proof}

\section{Examples} \label{examples}
The computations of the number of spanning trees in this section have been performed with the software \cite{SAGE}.  The computations of the approximated power series $Q_{a}(T)$ and $Q(T)$ have been performed with the software \cite{PARI}.  Examples are numerous and we present only a few.

\begin{enumerate}
\item Let us revisit example ($1$) of \S $5.6$ of \cite{Vallieres:2021}.  In this case, $\ell =2$, and we set $b_{1} = 1$.  The $1 \times 1$ matrix $M$ is given by
\begin{equation*}
\begin{pmatrix}
2X_{1} + 2Y_{1}
\end{pmatrix}.
\end{equation*}
We have $Q_{1}(T) = T$, and using Remark \ref{difptview}, we have
$$Q_{-1}(T) = -T - T^{2} - T^{3} - \ldots, $$
since $e_{k}(-1) = -1$ for all $k \ge 1$.  Thus, the power series $Q$ is
$$Q(T) = -2T^{2} - 2T^{3} -2T^{4}- \ldots \in \mathbb{Z}\llbracket T \rrbracket \subseteq  \mathbb{Z}_{2}\llbracket T \rrbracket,$$
so we should have $\mu = 1$ and $\lambda = 1$ which was indeed the case.  We point out in passing that when working with the power series $Q_{a}(T)$ of the current paper rather than the power series $P_{a}(T)$ of \cite{McGown/Vallieres:2021}, the power series $Q(T)$ is not a polynomial anymore.  The reason is because the power series $P_{a}(T)$ satisfy $P_{a}(T) = P_{-a}(T)$ and $P_{a}(T)$ is a polynomial when $a \in \mathbb{N}$, but the $Q_{a}(T)$ of the current paper do not satisfy this property.

\item Let us revisit example ($3$) of \S $5.6$ of \cite{Vallieres:2021}.  In this case, $\ell =3$, and we let $b_{1} = 1$, $b_{2} = 4$, and $b_{3}=20$.  The $1 \times 1$ matrix $M$ is given by
\begin{equation*}
\begin{pmatrix}
X_{1}+X_{2}+X_{3} +Y_{1}+Y_{2}+Y_{3}
\end{pmatrix}.
\end{equation*}
The power series $Q$ is given by
$$-417T^{2} -417 T^{3} -13737 T^{4} -27057T^{5} -215945T^{6}+ \ldots \in \mathbb{Z}\llbracket T \rrbracket \subseteq  \mathbb{Z}_{3}\llbracket T \rrbracket$$
The factorizations of the coefficients are
$$417 = 3 \cdot 139, \, 13737 = 3 \cdot 19 \cdot 241, \, 27057 = 3 \cdot 29 \cdot 311, \, 215945 = 5 \cdot 43189, $$
so we should have $\mu=0$ and $\lambda =5$ which was indeed the case.

\item Let us revisit example ($1$) of \S $5$ of \cite{McGown/Vallieres:2021}.  In this case, $\ell =2$, and we let $b_{1} = 1/3$ and $b_{2} = 3/5$.  The $1 \times 1$ matrix $M$ is given by
\begin{equation*}
\begin{pmatrix}
X_{1}+X_{2} +Y_{1}+Y_{2}
\end{pmatrix}.
\end{equation*}
The power series $Q(T) \in \mathbb{Z}_{2}\llbracket T \rrbracket$ is given by
$$(0.1101\ldots) T^{2} + (0.1101\ldots) T^{3} + (0.0011\ldots) T^{4} + (0.1011\ldots)T^{5} + (1.1100\ldots)T^{6}+ \ldots,$$
so we should have $\mu=0$ and $\lambda =5$ which was indeed the case.

\item Let us revisit (a) of example $7$ in \S $5.1$ of \cite{Gonet:2021}.  Let $X = K_{10}$ be the complete graph on ten vertices and let $\ell = 5$.  Label the vertices $V_{X} = \{ v_{1},\ldots,v_{10}\}$ and pick a section $\omega$ so that $s_{1}$ is the unique edge going from $v_{1}$ to $v_{2}$ and all edges are directed using the lexicographic order on $V_{X}$.  Consider the function $\alpha:S \longrightarrow \mathbb{Z}_{5}$ given by $\alpha(s_{1}) = 1$ and $\alpha(s_{i}) = 0$ if $i \neq 1$.  We let $b_{1} = 0$ and $b_{2}=1$.  The $10 \times 10$ matrix $M$ is given by
{\scriptsize
\begin{equation*}
\begin{pmatrix}
9 &-1 + X_{2} & -1+X_{1} & -1+X_{1} & -1+X_{1} & -1+X_{1} & -1+X_{1} & -1+X_{1} & -1+X_{1} & -1+X_{1}\\
-1+Y_{2} &9 & -1+X_{1} & -1+X_{1} & -1+X_{1} & -1+X_{1} & -1+X_{1} & -1+X_{1} & -1+X_{1} & -1+X_{1}\\
-1+Y_{1} & -1+Y_{1} & 9 & -1+X_{1} & -1+X_{1} & -1+X_{1} & -1+X_{1} & -1+X_{1} & -1+X_{1} & -1+X_{1}\\
-1+Y_{1} & -1+Y_{1} & -1+Y_{1} & 9 &-1+ X_{1} & -1+X_{1} & -1+X_{1} & -1+X_{1} &-1+ X_{1} & -1+X_{1} \\
-1+Y_{1} & -1+Y_{1} & -1+Y_{1} & -1+Y_{1} & 9 & -1+X_{1} & -1+X_{1} & -1+X_{1} & -1+X_{1} & -1+X_{1} \\
-1+Y_{1} & -1+Y_{1} & -1+Y_{1} & -1+Y_{1} & -1+Y_{1} & 9 &-1+ X_{1} & -1+X_{1} & -1+X_{1} & -1+X_{1} \\
-1+Y_{1} & -1+Y_{1} & -1+Y_{1} & -1+Y_{1} & -1+Y_{1} & -1+Y_{1} & 9 & -1+X_{1} & -1+X_{1} & -1+X_{1} \\
-1+Y_{1} & -1+Y_{1} & -1+Y_{1} & -1+Y_{1} & -1+Y_{1} & -1+Y_{1} & -1+Y_{1} & 9 &-1+ X_{1} & -1+X_{1} \\
-1+Y_{1} & -1+Y_{1} & -1+Y_{1} & -1+Y_{1} & -1+Y_{1} & -1+Y_{1} & -1+Y_{1} & -1+Y_{1} & 9 & -1+X_{1} \\
-1+Y_{1} & -1+Y_{1} & -1+Y_{1} & -1+Y_{1} & -1+Y_{1} & -1+Y_{1} & -1+Y_{1} & -1+Y_{1} & -1+Y_{1} & 9 \\
\end{pmatrix}
\end{equation*}}
\begin{flushleft}
The power series $Q$ starts as follows
$$Q(T) = -2^{10}\cdot 5^{7} \cdot \left(T^{2} + T^{3} + T^{4} + \ldots \right) \in \mathbb{Z}\llbracket T \rrbracket \subseteq \mathbb{Z}_{5}\llbracket T \rrbracket, $$
so we should have $\mu = 7$ and $\lambda=1$ which is what was obtained in \cite{Gonet:2021} as well.
\end{flushleft}

\item Let $\ell = 5$ and let $X$ be the dumbbell multigraph.  Let $\alpha:S \longrightarrow \mathbb{Z} \subseteq \mathbb{Z}_{5}$ be as in the example of \S \ref{voltage}, where we label the vertices of $X$ from left to right, so the left vertex is $v_{1}$ and the right vertex is $v_{2}$.  We let $b_{1}=1$, $b_{2}=0$, and $b_{3}=2$.  Then, we get:
\begin{equation*}
X \longleftarrow G(5,2) \longleftarrow G(5^2,2) \longleftarrow G(5^3,2) \longleftarrow \ldots,
\end{equation*}
where $G(n,k)$ denotes the generalized Petersen graph.  The $2 \times 2$ matrix $M$ is given by
\begin{equation*}
\begin{pmatrix}
1 + X_{1} + Y_{1} & -1 + X_{2}\\
-1 + Y_{2}& 1 + X_{3} + Y_{3}
\end{pmatrix}.
\end{equation*}
The power series $Q$ starts as follows 
$$Q(T) = -5T^{2} - 5 T^{3} -2 T^{4} + \ldots \in \mathbb{Z}\llbracket T \rrbracket \subseteq \mathbb{Z}_{5}\llbracket T \rrbracket,$$
so we should have $\mu = 0$ and $\lambda = 3$.  We calculate
$$\kappa_{0} = 1, \kappa_{1} = 2^{4}\cdot 5^{3}, \kappa_{2} = 2^{4} \cdot 5^{6} \cdot 480449^{2},$$
$$\kappa_{3} = 2^{4} \cdot 5^{9} \cdot 187751^{2} \cdot 480449^{2} \cdot 1829501^{2} \cdot 4731751^{2} \cdot 16390103451749^{2}, \ldots$$
We have
$${\rm ord}_{5}(\kappa_{n}) =  3n,$$
for all $n \ge 0$.
\item Let $\ell = 5$ and let $X$ be the dumbbell multigraph from \S \ref{voltage}.  We take the same section, and we label the vertices of $X$ from left to right as in the previous example.  We change $\alpha$ to be $\alpha:S \longrightarrow \mathbb{Z}_{5}$, where 
$$\alpha(s_{1}) = 2.11111\ldots \in \mathbb{Z}_{5}, $$
and $\alpha(s_{2}) = 0$, $\alpha(s_{3}) = 2$.  We let $b_{1} = 2.11111\ldots$, $b_{2} = 0$ and $b_{3}=2$.  Note that we have
$$p_{1}(b_{1}) = 2, p_{2}(b_{1}) = 7, p_{3}(b_{1}) = 32, \ldots $$
Then, we get an abelian $5$-tower
\begin{equation*}
X \longleftarrow I(5,2,2) \longleftarrow I(5^2,7,2) \longleftarrow I(5^3,32,2) \longleftarrow \ldots,
\end{equation*}
where $I(n,j,k)$ denotes the $I$-graph introduced in \cite{Foster:1988}.  For instance, $I(5,2,2)$ is
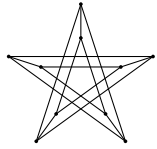
\begin{center}
\begin{tikzpicture}[baseline={([yshift=-0.6ex] current bounding box.center)}]
\node[draw=none,minimum size=2cm,regular polygon,regular polygon sides=5] (a) {};
\node[draw=none, minimum size=1.1cm,regular polygon,regular polygon sides=5] (b) {};

\foreach \x in {1,2,...,5}
  \fill (a.corner \x) circle[radius=0.7pt];
  
\foreach \y in {1,2,...,5}
  \fill (b.corner \y) circle[radius=0.7pt];
  
\path (a.corner 1) edge (a.corner 3);
\path (a.corner 2) edge (a.corner 4);
\path (a.corner 3) edge (a.corner 5);
\path (a.corner 4) edge (a.corner 1);
\path (a.corner 5) edge (a.corner 2);

\path (a.corner 1) edge (b.corner 1);
\path (a.corner 2) edge (b.corner 2);
\path (a.corner 3) edge (b.corner 3);
\path (a.corner 4) edge (b.corner 4);
\path (a.corner 5) edge (b.corner 5);

\path (b.corner 1) edge (b.corner 3);
\path (b.corner 1) edge (b.corner 4);
\path (b.corner 2) edge (b.corner 4);
\path (b.corner 2) edge (b.corner 5);
\path (b.corner 3) edge (b.corner 5);

\end{tikzpicture}
\end{center}
The $2 \times 2$ matrix $M$ is as in the previous example
\begin{equation*}
\begin{pmatrix}
1 + X_{1} + Y_{1} & -1 + X_{2}\\
-1 + Y_{2}& 1 + X_{3} + Y_{3}
\end{pmatrix}.
\end{equation*}
The power series $Q$ starts as follows
$$Q(T) = (2.4321\ldots)T^{2} + \ldots \in \mathbb{Z}_{5}\llbracket T \rrbracket, $$
so we should have $\mu = 0$ and $\lambda = 1$.  We calculate
$$\kappa_{0} = 1, \kappa_{1} = 5 \cdot 19^{2}, \kappa_{2} = 5^{2} \cdot 19^{2} \cdot 5340449^{2},$$
$$\kappa_{3} = 5^{3} \cdot 19^{2} \cdot 251^{2} \cdot 272999^{2} \cdot 5340449^{2} \cdot 60106497251^{2} \cdot 69266125271978251^{2}, \ldots$$
We have
$${\rm ord}_{5}(\kappa_{n}) =  n,$$
for all $n \ge 0$.

\item We modify the last example slightly.  Add one more loop at $v_{2}$ and denote it by $s_{4}$.  In this case, the base multigraph is not regular anymore.  Take $\alpha:S \longrightarrow \mathbb{Z}_{5}$ to be defined via
$$\alpha(s_{1}) = 2.1111111\ldots, \alpha(s_{2}) = 0, \alpha(s_{3}) = 2, \alpha(s_{4})  =1,$$  
and let $b_{1} = 2.11111\ldots$, $b_{2} = 0$, $b_{3}=2$, and $b_{4}=1$.  We get an abelian $5$-tower above the modified dumbbell multigraph.  The multigraph at the first layer is
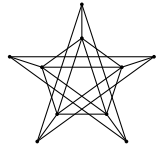
\begin{center}
\begin{tikzpicture}[baseline={([yshift=-0.6ex] current bounding box.center)}]
\node[draw=none,minimum size=2cm,regular polygon,regular polygon sides=5] (a) {};
\node[draw=none, minimum size=1.1cm,regular polygon,regular polygon sides=5] (b) {};

\foreach \x in {1,2,...,5}
  \fill (a.corner \x) circle[radius=0.7pt];
  
\foreach \y in {1,2,...,5}
  \fill (b.corner \y) circle[radius=0.7pt];
  
\path (a.corner 1) edge (a.corner 3);
\path (a.corner 2) edge (a.corner 4);
\path (a.corner 3) edge (a.corner 5);
\path (a.corner 4) edge (a.corner 1);
\path (a.corner 5) edge (a.corner 2);

\path (a.corner 1) edge (b.corner 1);
\path (a.corner 2) edge (b.corner 2);
\path (a.corner 3) edge (b.corner 3);
\path (a.corner 4) edge (b.corner 4);
\path (a.corner 5) edge (b.corner 5);

\path (b.corner 1) edge (b.corner 3);
\path (b.corner 1) edge (b.corner 4);
\path (b.corner 2) edge (b.corner 4);
\path (b.corner 2) edge (b.corner 5);
\path (b.corner 3) edge (b.corner 5);

\path (b.corner 1) edge (b.corner 2);
\path (b.corner 2) edge (b.corner 3);
\path (b.corner 3) edge (b.corner 4);
\path (b.corner 4) edge (b.corner 5);
\path (b.corner 1) edge (b.corner 5);

\end{tikzpicture}
\end{center}
The $2 \times 2$ matrix $M$ is 
\begin{equation*}
\begin{pmatrix}
1 + X_{1} + Y_{1} & -1 + X_{2}\\
-1 + Y_{2}& 1 + X_{3} + Y_{3} + X_{4} + Y_{4}
\end{pmatrix}.
\end{equation*}
The power series $Q$ starts as follows
$$Q(T) = (1.4321\ldots)T^{2} + \ldots \in \mathbb{Z}_{5}\llbracket T \rrbracket, $$
so we should have $\mu = 0$ and $\lambda = 1$.  We calculate
$$\kappa_{0} = 1, \kappa_{1} = 5 \cdot 71^{2}, \kappa_{2} = 5^{2} \cdot 71^{2} \cdot 101^{2} \cdot 1399^{2} \cdot 20749^{2},$$
and $\kappa_{3}$ is
{\footnotesize $$5^{3} \cdot   71^{2} \cdot 101^{2} \cdot 1399^{2} \cdot 9749^{2} \cdot 20749^{2} \cdot 2144501^{2} \cdot 5100751^{2} \cdot 79219966014816908153269755910249^{2}$$} 
\begin{flushleft}
We have
$${\rm ord}_{5}(\kappa_{n}) =  n,$$
for all $n \ge 0$.
\end{flushleft}

\item Let $\ell = 3$, and consider now the following multigraph $X$
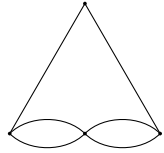
\begin{center}
\begin{tikzpicture}
\fill (0,0) circle[radius=0.7pt];
\fill (1,0) circle[radius=0.7pt];
\fill (2,0) circle[radius=0.7pt];
\fill (1,1.73) circle[radius=0.7pt];
\draw (0,0) .. controls (0.25,0.25) and (0.75,0.25)  .. (1,0);
\draw (0,0) .. controls (0.25,-0.25) and (0.75,-0.25)  .. (1,0);
\draw (1,0) .. controls (1.25,0.25) and (1.75,0.25)  .. (2,0);
\draw (1,0) .. controls (1.25,-0.25) and (1.75,-0.25)  .. (2,0);
\draw (2,0) -- (1,1.73) -- (0,0);
\end{tikzpicture}
\end{center}
\begin{flushleft}
and the section:
\end{flushleft}
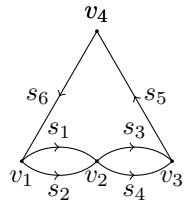
\begin{center}
\begin{tikzpicture}
\fill (0,0) circle[radius=0.7pt];
\fill (1,0) circle[radius=0.7pt];
\fill (2,0) circle[radius=0.7pt];
\fill (1,1.73) circle[radius=0.7pt];

\begin{scope}[decoration={
    markings,
    mark=at position 0.5 with {\arrow{>}}}
    ] 

\draw[postaction={decorate}](0,0) node[below]{$v_1$}.. controls (0.25,0.25) and (0.75,0.25)  .. node[above] {$s_1$} (1,0) node[below]{$v_2$};
\draw[postaction={decorate}] (0,0) .. controls (0.25,-0.25) and (0.75,-0.25)  .. node[below] {$s_2$} (1,0);
\draw[postaction={decorate}] (1,0) .. controls (1.25,0.25) and (1.75,0.25)  .. node[above] {$s_3$} (2,0) node[below]{$v_3$};
\draw[postaction={decorate}] (1,0) .. controls (1.25,-0.25) and (1.75,-0.25)  .. node[below] {$s_4$} (2,0);
\draw[postaction={decorate}] (2,0) -- node[right]{$s_5$} (1,1.73) node[above]{$v_4$};
\draw[postaction={decorate}] (1,1.73) node[above]{$v_4$} -- node[left]{$s_6$} (0,0);

\end{scope}

\end{tikzpicture}
\end{center}
Consider the function $\alpha:S \longrightarrow \mathbb{Z} \subseteq \mathbb{Z}_{3}$ defined via
$$\alpha(s_{3}) = \alpha(s_{6})=1,  \alpha(s_{4})=2, \alpha(s_{1})=\alpha(s_{2}) = \alpha(s_{5}) = 0. $$
We then obtain an abelian $3$-tower over $X$:
$$X  \longleftarrow X(\mathbb{Z}/3 \mathbb{Z},S,\alpha_{1}) \longleftarrow X(\mathbb{Z}/3^{2}\mathbb{Z},S,\alpha_{2}) \longleftarrow \ldots \longleftarrow X(\mathbb{Z}/3^{n}\mathbb{Z},S,\alpha_{n}) \longleftarrow \ldots$$
The $4 \times 4$ matrix $M$ is given by
\begin{equation*}
\begin{pmatrix}
3 & -2 + 2X_{1} & 0 & -1+Y_{2}\\
-2 + 2Y_{1} & 4 & -2+X_{2}+X_{3} & 0 \\
0 & -2 + Y_{2} + Y_{3} & 3 & -1+ X_{1} \\
-1+X_{2} & 0 & -1 + Y_{1} & 2
\end{pmatrix}.
\end{equation*}
The power series $Q$ starts as follows
$$-31T^{2} - 31T^{3} - 45T^{4} - \ldots \in \mathbb{Z}\llbracket T\rrbracket \subseteq \mathbb{Z}_{3} \llbracket T\rrbracket,$$
so we should have $\mu=0$ and $\lambda=1$.  We calculate
$$\kappa_{0} = 2^{2} \cdot 3, \kappa_{1} = 2^{2} \cdot 3^{2} \cdot 7^{2}, \kappa_{2} = 2^{2} \cdot 3^{3} \cdot 7^{2} \cdot 2251^{2}, \ldots $$
We have
$${\rm ord}_{3}(\kappa_{n}) = n + 1, $$
for all $n \ge 0$.
\end{enumerate}

\section{Iwasawa modules arising from graph theory} \label{Gonet}
We end this paper with the following comment.  Let $\Gamma$ be a multiplicative topological group isomorphic to $\mathbb{Z}_{\ell}$.  For $n \ge 1$, let $\Gamma_{n} = \Gamma/\, \Gamma^{\ell^{n}} \simeq \mathbb{Z}/\ell^{n}\mathbb{Z}$, and consider the classical Iwasawa algebra
$$\Lambda = \lim_{\substack{\longleftarrow \\ n \ge1 }}\mathbb{Z}_{\ell}[\Gamma_{n}]. $$
It was shown in \cite{Serre:1995} that there is an isomorphism
$$\mathbb{Z}_{\ell}\llbracket T\rrbracket  \stackrel{\simeq}{\longrightarrow} \Lambda,$$
given by $T \mapsto 1-\gamma$, where $\gamma$ is a topological generator of $\Gamma$.  There is a known structure theorem for finitely generated $\Lambda$-modules up to pseudo-isomorphisms (see for instance \cite{Washington:1997}) and to every such finitely generated $\Lambda$-module $N$ is associated its Iwasawa invariants $\mu(N)$ and $\lambda(N)$.  In light of Theorem \ref{maintheorem}, it is natural to ask if there would be some finitely generated $\Lambda$-modules arising from graph theory.  If 
$$X = X_{0} \longleftarrow X_{1} \longleftarrow X_{2} \longleftarrow \ldots \longleftarrow X_{n} \longleftarrow \ldots$$
is an abelian $\ell$-tower of multigraphs, let $\mathcal{A}_{n}$ be the Sylow $\ell$-subgroup of the Picard group ${\rm Pic}^{\circ}(X_{n})$.  Recall that the cardinality of ${\rm Pic}^{\circ}(X_{n})$ is $\kappa_{n}$ and thus
$$|\mathcal{A}_{n}| = \ell^{{\rm ord}_{\ell}(\kappa_{n})}.$$  
(See the introduction to \cite{Lorenzini:2008} for an overview of where the finite abelian group ${\rm Pic}^{\circ}(X_{n})$ arises in mathematics and under what names.)  If the base graph $X$ is a simple graph (meaning it does not contain loops and parallel edges), Gonet constructed a $\Lambda$-module that is finitely generated and for which one can recapture the $\mathcal{A}_{n}$ by looking at some of its quotients.  This allowed her to obtain the existence of the invariants $\mu$ and $\lambda$ of Theorem \ref{maintheorem} when the base multigraph is a simple graph.  (See Theorem $27$ of \cite{Gonet:2021}.)  Theorem \ref{maintheorem} therefore gives in particular a way to calculate the Iwasawa invariants introduced in her thesis.

\bibliographystyle{plain}
\bibliography{main}

\end{document}